\documentclass[twoside,reqno]{amsart}

\usepackage[english]{babel} 
\usepackage{exscale}   
\usepackage{amsmath}
\usepackage{amsfonts}
\usepackage{amssymb}
\usepackage{amsxtra}
\usepackage{stmaryrd}
\usepackage{xspace}
\usepackage{epsfig}
\usepackage{mathrsfs}
\usepackage{pstricks}
\usepackage{enumitem}
\usepackage{latexsym}
\usepackage{graphics}
\usepackage{xypic}

\newtheorem{theorem}{Theorem}[section]
\newtheorem{theoremanddefinition}[theorem]{Theorem and Definition}
\newtheorem{proposition}[theorem]{Proposition}

\newtheorem{lemma}[theorem]{Lemma}
\newtheorem{definition}[theorem]{Definition}
\newtheorem{example}[theorem]{Example}
\newtheorem{remark}[theorem]{Remark}

\newcommand{\Bset}{\mathbb{B}}
\newcommand{\Cset}{\mathbb{C}}

\newcommand{\Nset}{\mathbb{N}}

\newcommand{\Sset}{\mathbb{S}}

\newcommand{\Zset}{\mathbb{Z}} 
\newcommand{\CA}{\ensuremath{{\mathcal A}}\xspace}         
\newcommand{\CB}{\ensuremath{{\mathcal B}}\xspace}         
\newcommand{\CC}{\ensuremath{{\mathcal C}}\xspace}         
\newcommand{\CM}{\ensuremath{{\mathcal M}}\xspace}         
\newcommand{\CN}{\ensuremath{{\mathcal N}}\xspace}         
\newcommand{\eins}{\ensuremath{{\rm 1\kern-.25em l}}\xspace}  
 
\newcommand{\id}{\operatorname{id}}                        

\newcommand{\trace}{\operatorname{tr}}        

\newsavebox{\artin}
\newcommand{\masterartin}{
\linethickness{1pt}
\qbezier(10,20)(0,10)(-10,0)
\qbezier(-10,20)(-7,17)(-4,14)
\qbezier(10,0)(6,4)(4,6)
}
\newsavebox{\artininv}
\newcommand{\masterartininv}{
\linethickness{1pt}
\qbezier(-10,20)(0,10)(10,0)
\qbezier(10,20)(6,16)(4,14) 
\qbezier(-10,0)(-5,5)(-4,6)
}
\newsavebox{\strandr}
\newcommand{\masterstrandr}{
\linethickness{1pt}
\qbezier(10,20)(10,10)(10,0)
}
\newsavebox{\strandl}
\newcommand{\masterstrandl}{
\linethickness{1pt}
\qbezier(-10,20)(-10,10)(-10,0)
}
\newsavebox{\horizontaldots}
\newcommand{\masterhorizontaldots}{
\linethickness{1pt}
\put(5,0){\circle*{2}}
\put(11,0){\circle*{2}}
\put(17,0){\circle*{2}}
}

\begin{document}

\title{Semi-Cosimplicial Objects and Spreadability}
\author{D.~Gwion Evans}
\address{D.~Gwion Evans \\
Department of Mathematics \\ 
Aberystwyth University \\
Aberystwyth, SY23 3BZ, UK}
\email{dfe@aber.ac.uk}
\author{Rolf Gohm}
\address{Rolf Gohm \\
Department of Mathematics \\ 
Aberystwyth University \\
Aberystwyth, SY23 3BZ, UK}
\email{rog@aber.ac.uk}
\author{Claus K\"{o}stler}
\address{Claus K\"{o}stler\\
School of of Mathematical Sciences \\ 
University College Cork\\
Cork, Ireland}
\email{Claus@ucc.ie}
\subjclass[2000]{}
\keywords{}
\date{26 February 2016}
\begin{abstract}
To a semi-cosimplicial object (SCO) in a category we associate a system of partial shifts on the inductive limit. We show how to produce an SCO from an action of the infinite braid monoid $\Bset^+_\infty$ and provide examples. In categories of (noncommutative) probability spaces SCOs correspond to spreadable sequences of random variables, hence SCOs can be considered as the algebraic structure underlying spreadability.
\end{abstract}
\maketitle

\section{Introduction}
\label{section:intro}

Distributional symmetries have been intensely studied in probability theory in recent decades, see the monograph \cite{Ka05} for an inspiring overview. More recently it emerged that distributional symmetries are also crucial for the further development of noncommutative probability theory and that an important role is played by a specific distributional symmetry (or invariance principle) which is called spreadability, i.e., the invariance of distribution if one passes from a sequence of random variables to a subsequence. See \cite{Ko10} for the beginning of this story. 
One can argue that these symmetries become more transparent from an algebraic point of view if we interpret probability theory as a study of associative algebras and their states and so the point of view of noncommutative probability theory is a natural one.  
In this paper we deepen these connections to algebra by including concepts from category theory and homological algebra. Not only do we obtain a better idea of what spreadability really means and arrive at the natural level of generality for constructing further examples. We also arrive at the fundamental insight that unlike other probabilistic symmetries which are based on group actions spreadability really has a homological flavour. To make this completely explicit is one of the main targets of this paper.

What we need to study for this purpose are semi-cosimplicial objects (SCOs for short). We briefly recollect the relevant concepts, see for example \cite{We94}, Chapter 8.1, for more details. Some of the most fundamental ideas of algebraic topology and homological algebra relate to simplices and they can be based on the simplicial category $\Delta$. The objects of $\Delta$ are finite ordered sets, usually written as $[n] := \{0,1,\ldots,n\},\; n \in \Nset_0$, and the morphisms are all non-decreasing maps between these objects. An interesting subcategory $\Delta_S$, called the semi-simplicial category, is obtained by considering the same objects but only (strictly) increasing maps as morphisms. 
Other names in use for this important category $\Delta_S$ are `restricted simplicial' \cite{BS14} and `incomplete simplicial' \cite{Mi03} (see \cite[8.1.10]{We94} for historical remarks about the terms).
In this paper only the semi-simplicial category $\Delta_S$ is relevant and so we give further definitions only in this context. Let us remark however that it is always interesting to ask if constructions actually can be extended to the simplicial category $\Delta$ in some way.

A covariant functor $F$ from the semi-simplicial category $\Delta_S$ to another category $\CC$ is called a semi-cosimplicial object (SCO) in $\CC$. We can work out a more explicit description of what a SCO is by noting that the morphisms of $\Delta_S$ are generated by the face maps
\[
\delta^k \colon [n-1] \rightarrow [n],\quad
m \mapsto m \; \text{if}\;\; m<k,\; 
m \mapsto m+1 \; \text{if}\;\; m \ge k.
\]

Here $k=0,\ldots,n$ and $n \in \Nset$. Following the usual convention we omit the index $n$ in the notation
of the $\delta^k$ and leave the domain and codomain to the context. The $\delta^k$ satisfy the cosimplicial identities
\[
\delta^j \delta^i = \delta^i \delta^{j-1} \quad \text{if}\,\;i < j
\]
and these cosimplicial identities provide a presentation of the category $\Delta_S$. The functor $F$ takes $[n]$ to $F[n]$ and $\delta^k$ to $F(\delta^k): F[n-1] \rightarrow F[n]$. Simplifying the notation we can then write
$F^n$ for $F[n]$ and $\delta^k$ for $F(\delta^k)$ and obtain the explicit definition of an SCO to be used in the sequel.

\begin{definition} \label{def:cosimp}
A semi-cosimplicial object (SCO) in the category $\CC$ is a sequence $(F^n)_{n\in \Nset_0}$ of objects in $\CC$ together with morphisms (coface operators)
\[
\delta^k: F^{n-1} \rightarrow F^n \quad\quad (k=0,\ldots,n)
\]
satisfying the cosimplicial identities
\[
\delta^j \delta^i = \delta^i \delta^{j-1} \quad \text{if}\,\;i < j\,.
\]
If there is an additional object $F^{-1}$ in $\CC$ together with a morphism $\delta^0: F^{-1} \rightarrow F^0$ satisfying the cosimplicial identities then we have an augmented semi-cosimplicial object.
\end{definition}

We refer to, for example, \cite{We94, Sm01, Wu10} for more information about (semi co-)simplicial objects and for a development of the rich and far developed theory built around them. 
\\

In this paper we proceed as follows.
In Section \ref{section:adapted} we develop some category theory which provides a general framework. In particular we show that by forming an inductive limit from a given SCO we get in addition a sequence of adapted endomorphisms with properties reflecting the SCO. We call this an SCO-system of partial shifts and we study some of its properties. The most basic example of an SCO-system of partial shifts (providing a good guide for the intuition) appears on the set $\Nset_0$ of nonnegative integers and it consists of the sequence of maps $\alpha_k \colon \Nset_0 \rightarrow \Nset_0$ (with $k \in \Nset_0$) given by $\alpha_k (m) := m$ if $m < k$ and $\alpha_k (m) := m+1$ if $m \ge k$ (missing the position $k$). Note the close similarity to the face maps described above. 
We should emphasize that SCO-systems of partial shifts are nothing but a convenient tool to handle SCOs, in particular they are a useful bridge to the probabilistic contexts studied later, but with the techniques of Section \ref{section:adapted} it is possible to give formulations directly in terms of the SCO if this is preferred. We also include in Section \ref{section:adapted} some examples of semi-cosimplicial groups which can be constructed in an elementary way.

In Section \ref{section:braid} we study a way to construct SCOs from actions of the infinite braid monoid $\Bset^+_\infty$. This generalizes the idea of braidability in \cite{GK09} (which is discussed briefly later in this introduction) and gives a somewhat simplified way of thinking about this concept. To our knowledge this is also a new way to create cosimplicial identities as needed for an SCO and it creates a wealth of nontrivial examples worth of further study. For instance there is the corresponding standard semi-cosimplicial cohomology theory which we mention briefly in Remark \ref{rem:cohom} but do not investigate further in this paper. 

In Section 4 we first recall the definition of spreadability in various categories of (noncommutative) probability spaces. Then we can state Theorem \ref{thm:spread} which (together with Theorem \ref{thm:spread*}) is our main result. It states that SCOs in these categories induce spreadable sequences of random variables and, conversely, the distribution of a spreadable sequence can always be achieved from an SCO in such a category. 
We only develop the most basic part of the theory here but it should already be enough to convince the reader that SCOs are the fundamental algebraic structure underlying spreadability. 

To guide the reader's intuition through the paper let us insert already here an example in the category of unital associative algebras (or $*$-algebras) which is fundamental in many ways. Let $\CB$ be such an algebra. Then we can form an SCO $(X^n)_{n\in\Nset_0}$ with tensor products $X^n := \bigotimes^n_0 \CB$ together with coface operators
\[
\delta^k: \bigotimes^{n-1}_0 \CB \rightarrow \bigotimes^n_0 \CB, \quad x_0 \otimes \ldots \otimes x_{n-1} 
\quad\mapsto\quad x_0 \otimes \ldots \otimes x_{k-1} \otimes \eins \otimes x_k \otimes \ldots \otimes x_{n-1}.
\]

The cosimplicial identities can easily be checked directly in this case. Alternatively the reader who has studied 
Section \ref{section:braid} is invited to work out that this is a special case of the theory presented in 
Theorem \ref{thm:braid}. In fact it comes from the braid group representation factoring through the representation of the symmetric group which permutes the tensor products.

We can also illustrate the theory of SCO-systems of partial shifts from Section \ref{section:adapted} with this example. There is an inductive system of tensor products of the algebra $\CB$ with itself, with inclusions $x \mapsto x \otimes \eins$ and inductive limit $\CA := \bigotimes^\infty_0 \CB$. Hence on $\CA$ we have the tensor shift $\alpha_0: \CA \rightarrow \CA$ and we get the canonically associated SCO-system of partial shifts by considering the sequence of algebra homomorphisms $\alpha_k \colon \CA \rightarrow \CA$ (with $k \in \Nset_0$) given by
$\alpha_k (x) := x$ if $x \in \bigotimes^{k-1}_0 \CB$ but $\alpha_k (x) := \alpha_0(x)$ if $x \in \bigotimes^\infty_k \CB$. 

By additionally choosing a unital linear functional invariant under all these partial shifts we finally arrive in our example at the theory of spreadability in (noncommutative) probability spaces. The basic example is to choose any unital linear functional $\varphi_\CB$ on $\CB$ and then to construct the infinite tensor product $\varphi := \bigotimes^\infty_0 \varphi_\CB$. Convex combinations of such products provide more examples. 
The reader familiar with the notion of spreadability (which we review in Section \ref{section:spread}) will have no difficulty to verify that under these circumstances the embeddings of the 
noncommutative probability space $(\CB, \varphi_\CB)$ into the different positions of the tensor product provide an example of a spreadable sequence.  

To make some finer distinctions we decided in this paper to refer to the version of spreadability based on $*$-algebras and $*$-homomorphisms, most relevant for the probabilistic point of view, as $*$-spreadable.
We provide the interesting Example \ref{example:TL} where, in contrast to the tensor product example given above, no simplification based on the more traditional idea of exchangeability is possible and the full strength of the results in Section \ref{section:braid} is needed. Namely we construct spreadable sequences of operators and, in particular, of projections in the tower associated to a subfactor, in the theory of von Neumann algebras (see \cite{GHJ89}). Spreadability follows for all values of the Jones index but $*$-spreadability only appears if the index is small (i.e., $\le 4$).

Applications of spreadability are not the topic of this paper but let us finish this introduction with a short review of the literature, including some of our motivations and some background why in particular the study of $*$-spreadability is important and relevant for a probabilist.
In fact, in the $*$-algebra setting, where we use states instead of unital linear functionals, the tensor product example above is the basis of what probabilists call exchangeability, which is an important special case of spreadability. It is intimately connected to the representation of the symmetric group mentioned above. 

Clearly the category of $*$-probability spaces includes classical probability in the sense that we can consider a commutative $*$-algebra of complex functions on a classical probability space and a positive functional induced by a probability measure. In this case spreadable and $*$-spreadable is the same property. To avoid technical difficulties in the following discussion let us always assume that we have 
Lebesgue spaces, so we can, for example, represent homomorphisms of the measure algebras by point transformations (modulo sets of measure zero). See \cite{Pet83}, 1.4C for more details and further references.

Here the interest in spreadability comes from the fact that a de Finetti type theorem can be proved, i.e., we can deduce a form of conditional independence. Using the background and terminology provided in \cite{Ka05} we have 

\begin{theorem} \label{thm:classical}
Let $(\xi_n)_{n\in\Nset_0}$ be a sequence of (classical) random variables (realised by measure-preserving maps between Lebesgue spaces).
Further let $\Sigma_\infty$ denote the $\sigma$-algebra generated by the sequence 
$(\xi_n)_{n\in\Nset_0}$ and $\Sigma_n$ the $\sigma$-algebra generated by $\xi_0, \ldots, \xi_n$, for all $n \in \Nset_0$.
Then the following are equivalent:
\begin{itemize}
\item[(a)]
$(\xi_n)$ is spreadable.
\item[(b)]
$(\xi_n)$ is exchangeable.
\item[(c)]
$(\xi_n)$ is conditionally i.i.d. (independent and identically distributed).
\item[(d)]
For all $n \in \Nset$ there exist 
$\Sigma_n\!-\!\Sigma_{n-1}$-measurable and measure-preserving maps $\delta_k$,
for $k=0,\ldots,n$, such that (with $0 \le N < n$) we have $\xi_N = \xi_N \circ \delta_k$ for $N < k$ and $\xi_{N+1} = \xi_N \circ \delta_k$ for $N \ge k$.
\item[(e)]
There exist $\Sigma_\infty$-measurable measure-preserving maps 
$(\beta_k)_{k \ge 0}$ such that we have $\xi_N = \xi_N \circ \beta_k$ for $N < k$ and $\xi_{N+1} = \xi_N \circ \beta_k$ for $N \ge k$.
\end{itemize}
\end{theorem}

We do not discuss exchangeability further in this paper and recommend \cite{GK09, Ko10, GK12} for further results about exchangeability from our point of view. 
The equivalence of (a), (b) and (c) is provided by Theorem 1.1 in \cite{Ka05} where not only a proof but a lot of further information can be found (see also \cite{Ka88}). In fact, among other things it is shown there how from spreadability one can obtain a very transparent proof of the classical de Finetti theorem (which is the equivalence with (c)) via the mean ergodic theorem.
The equivalence of (a) and (d) is exactly the topic of this paper, applied to this specific situation. In fact, it is a special case of the equivalence of (2)(a) and (2)(c) in our Theorem \ref{thm:spread} (or  Theorem \ref{thm:spread*}) below. 
More explicitly, it follows from spreadability that omitting $\xi_k$ from
$\xi_0, \ldots, \xi_n$ yields the same distribution as does $\xi_0, \ldots, \xi_{n-1}$ and this allows us to define the measure preserving transformation $\delta_k$. Conversely, if we start from such $\delta_k$, then by $\delta^k p := p \circ \delta_k$ on polynomials
p in the random variables $\xi_0, \ldots, \xi_{n-1}$ we obtain an SCO
as in Theorem \ref{thm:spread} (2)(c) and deduce spreadability from that as described there.
Note that in (d) we have used again the convention that the dependence on $n$ of the maps $\delta_k$ is suppressed in notation. Equivalently, by the theory we develop in Section \ref{section:adapted}, we have the formulation in (e) involving the measure-preserving maps version of what we call partial shifts.
\\

Let us go back to the general noncommutative setting.
The case of $*$-spreadability for noncommutative random variables in von Neumann algebras arising from actions of the infinite braid group 
$\Bset_\infty$ by state-preserving $*$-automorphisms was investigated extensively in \cite{GK09} and the sequences obtained in this way were called braidable there (we call them $*$-braidable here). The fact that $*$-braidable sequences are $*$-spreadable, obtained in \cite{GK09}, follows again from our Theorem \ref{thm:braid} together with Theorem \ref{thm:spread*}. The converse question: "Is every $*$-spreadable sequence necessarily 
also $*$-braidable?", seems to be still open at the moment. In fact, this open question was one of the motivations for this paper. Our expectation is that the characterization of $*$-spreadability in terms of SCOs achieved in Theorem \ref{thm:spread*} will provide a tool to construct examples which show that the answer is negative. 

It was shown in \cite{Ko10} that in the setting of (in general noncommutative) von Neumann algebras and corresponding noncommutative probability spaces there is still a version of de Finetti's theorem for $*$-spreadable sequences which makes use of a generalized notion of noncommutative stochastic independence. The proof involves refined applications of the mean ergodic theorem and it is in this context that the idea of partial shifts first appeared (which we take up in Section \ref{section:adapted} and derive it from SCOs). 
Moreover the braidability results in \cite{GK09} show that in the noncommutative setting $*$-spreadability is much more general than $*$-exchangeability
(which involves representations of the infinite symmetric group while representations of the infinite braid group are sufficient to produce spreadability, as explained above). Hence there are many indications that in noncommutative probability theory the notion of spreadability is actually more fundamental than exchangeability.

We finally mention the notion of quantum spreadability developed in \cite{Cu11} which strengthens the notion of spreadability using the idea of quantum increasing sequence spaces. It is shown in \cite{Cu11} that quantum spreadability is equivalent with free independence, hence is strong enough to enforce a very specific structure for the noncommutative probability space. In contrast, in this paper we consider the (weaker) classical notion of spreadability but, in general, we apply it to noncommutative probability spaces as well. Here we find that this does not enforce a specific structure for the noncommutative probability space but instead yields interesting general results (for example the de Finetti type results) for a wide range of such spaces and as such is worth of further study. 

We expect that the clear identification of SCOs as the algebraic backbone of spreadability obtained here will lead to the construction of further examples and to new theoretical developments.

\section{Adaptedness. SCOs and Partial Shifts}
\label{section:adapted}
We start by giving a definition of adaptedness in terms of category theory and derive a global formulation for SCOs by so-called SCO-systems of partial shifts. 
Later this helps us to describe the connection between the (co)simplicial theory and probability theory in a flexible and convenient way. With slight modifications we follow here the approach in \cite{Go04}, Section 3.2. 
For category theory itself we follow \cite{ML98}. Consider the category
$\omega = \{0 \rightarrow 1 \rightarrow 2 \rightarrow 3 \rightarrow \ldots\}$
and another category $\CC$ which allows $\omega-$colimits (= inductive limits). We briefly recall what this means.
Suppose that $F \colon \omega \rightarrow \CC$ is a functor, i.e., we have
\[
F_0 \stackrel{i_1}{\rightarrow} F_1 \stackrel{i_2}{\rightarrow} F_2 \stackrel{i_3}{\rightarrow} \ldots
\]
with morphisms $i_n: F_{n-1} \rightarrow F_n,\;n \in \Nset,$ between $\CC$-objects. We also refer to this functor as a filtration (by a slight abuse of language the sequence of objects is also called a filtration).
An $\omega-$colimit (or inductive
limit) is an object $F_\infty = \lim\limits_{\rightarrow} F$ in $\CC$ which together with canonical arrows 
$\mu_n: F_n \rightarrow F_\infty$ ($n\geq 0$)
forms a universal cone (see \cite{ML98}, III.3). Pictorially this is a commuting diagram:
\[
\xymatrix@R30pt@C=70pt{
F_0 \ar[r]^{i_1} \ar[dr]^{\mu_0} \ar[ddr]_{\mu^\prime_0}
& F_1 \ar[r]^{i_2} \ar[d]_{\mu_1}
& F_2 \ar[r]^{i_3} \ar[dl]_{\mu_2} \ar[ddl]^{\mu^\prime_2}
& \ldots \\
 & F_\infty \ar@{->}[d]^{\exists ! f} & & \\
 & F^\prime 
}
\]
($\mu^\prime_1$ not drawn).
So the $\CC$-object $F_\infty$ is determined up to isomorphism by the fact that there are morphisms
$\mu_n: F_n \rightarrow F_\infty,\;n \in \Nset_0,$ which satisfy the equations $\mu_n\, i_n = \mu_{n-1},\;n \in \Nset,$
and are universal with respect to any morphisms
$\mu^\prime_n: F_n \rightarrow F^\prime,\;n \in \Nset_0,$ which satisfy the equations $\mu^\prime_n\, i_n = \mu^\prime_{n-1},\;n \in \Nset$. In many examples these morphisms involve inclusions of sets but this is not necessarily the case in general.

\begin{lemma} \label{lem:adapted}
Given morphisms $\alpha^{(n)}: F_{n-1} \rightarrow F_n,\;n \in \Nset$, such that
\[
\mu_{n+1}\, \alpha^{(n+1)}\, i_n = \mu_n \,\alpha^{(n)} \quad (\text{for all}\;\, n \in \Nset)
\]
then there exists a unique morphism $\alpha: F_\infty \rightarrow F_\infty$ such that
\[
\alpha \, \mu_{n-1} = \mu_n \,\alpha^{(n)} \quad (\text{for all}\;\, n \in \Nset).
\]
\end{lemma}

\begin{proof}
If we define $\mu^\prime_n := \mu_{n+1}\,\alpha^{(n+1)}$ then
\[
\mu^\prime_n\,i_n = \mu_{n+1}\,\alpha^{(n+1)}\,i_n = \mu_n \,\alpha^{(n)} = \mu^\prime_{n-1},
\]
and we get $\alpha$ from the universal property ($\alpha \mu_{n-1} = \mu^\prime_{n-1}$).
\end{proof}

\begin{definition} \label{def:adapted}
A morphism $\alpha: F_\infty \rightarrow F_\infty$ given as in Lemma \ref{lem:adapted} is called an adapted endomorphism (with respect to the filtration) determined by $(\alpha^{(n)})_{n \in \Nset}$.
\end{definition}

Intuitively $\alpha^{(n)}$ describes how $\alpha$ acts on (the image of) the $(n-1)$-th object in the filtration and adaptedness describes the compatibility of these actions. The terminology is motivated by stochastic processes and their time evolutions, cf. Section \ref{section:spread}. 

\begin{lemma} \label{lem:trivial}
Let $\alpha$ be an adapted endomorphism (with respect to a filtration). If $\alpha \mu_n = \mu_n$ for some $n$ then also
$\alpha \mu_k = \mu_k$ for all $k \le n$.
\end{lemma}

\begin{proof}
If $\alpha \mu_n = \mu_n$ for some $n$ then
\[
\alpha \mu_{n-1} = \mu_n \alpha^{(n)} 
= \mu_{n+1} \alpha^{(n+1)} i_n = \alpha \mu_n i_n
= \mu_n i_n = \mu_{n-1}\,.
\]
By iterating this argument we get the stated result.
\end{proof}

If $\alpha$ satisfies the condition of Lemma \ref{lem:trivial} we say that $\alpha$ acts trivially on (the image of) the $n$-th object. 

\begin{definition} \label{def:partial}
Let $(\alpha_k)_{k\in\Nset_0}$ be a sequence of adapted endomorphisms (with respect to a common filtration). If the sequence satisfies
\begin{itemize}
\item[(1)]
for each $k \in \Nset$ the endomorphism $\alpha_k$ acts trivially on (the image of) the $(k-1)$-th object,
\item[(2)]
$\alpha_j \alpha_i = \alpha_i \alpha_{j-1}$ if $i,j \in \Nset_0$
and $i<j$,
\end{itemize}
then we say that $(\alpha_k)_{k\in\Nset_0}$ is an SCO-system of partial shifts (for this filtration). 
\end{definition}

Note that if $(\alpha_k)_{k\in\Nset_0}$ is an SCO-system of partial shifts then for all $\ell \in \Nset$ the sequence $(\alpha_k)_{k \ge \ell}$ is also an SCO-system of partial shifts if everything is suitably relabeled ($k \mapsto k-\ell$).

\begin{proposition} \label{prop:partial}
Let $(\alpha_k)_{k \in \Nset_0}$ be an SCO-system of partial shifts. 
Then
\[
\alpha_k (\alpha_0)^N \mu_0 =
\left\{
\begin{array}{cc}
(\alpha_0)^N\; \mu_0 & \text{if}\;\; N < k \\ 
(\alpha_0)^{N+1} \mu_0 & \text{if}\;\; N \ge k \\
\end{array} 
\right.
\]
\end{proposition}

\begin{proof}
If $N < k$ then, using properties (2) and (1) of partial shifts
\[
\alpha_k (\alpha_0)^N \mu_0 = (\alpha_0)^N \alpha_{k-N} \mu_0
= (\alpha_0)^N \mu_0\,.
\]
If $N \ge k$ then, using property (2) of partial shifts
\[
\alpha_k (\alpha_0)^N \mu_0 = (\alpha_0)^{k} \alpha_0 (\alpha_0)^{N-k} \mu_0 = (\alpha_0)^{N+1} \mu_0\,.
\]
\end{proof}

Proposition \ref{prop:partial} explains the terminology of partial shifts:
we regard $\alpha_0$ as a full shift while $\alpha_k$ for $k \ge 1$ 
acts trivially on an initial part and only shifts the remaining part. We explain the origin of this concept in the theory of spreadability further in Section \ref{section:spread}. Of course the property in Proposition \ref{prop:partial} also reminds us of the origin of coface operators from face maps, i.e., specific strictly increasing functions missing one point, mentioned in Section \ref{section:intro}, so we have come full circle. Note that Proposition \ref{prop:partial} applied to the relabeled SCO-systems (as mentioned above) gives additional relationships.

The following theorem, the main result of this section, gives a correspondence between SCOs and SCO-systems of partial shifts.
While we usually suppress the covariant functor $F$ corresponding to an SCO in the notation, in this argument we write it down to make the interplay with the inductive limit construction explicit. 

\begin{theoremanddefinition} \label{thm:correspondence}
(a) Let a covariant functor $F$ from the semi-simplicial category $\Delta_S$ to a category $\CC$ be given, with the corresponding SCO in $\CC$ described by $F[n],\; F(\delta^k),
\;k=0,\ldots,n$ and $n\in\Nset_0$. We can restrict to a functor from $\omega$ to $\CC$ (which we also denote by $F$), given by
\[
F[0] \stackrel{i_1}{\rightarrow} F[1] \stackrel{i_2}{\rightarrow} F[2] \stackrel{i_3}{\rightarrow} \ldots
\]
where $i_n := F(\delta^n): F[n-1] \rightarrow F[n]$ for $n \in \Nset$. If there exists an $\omega$-colimit
$F_\infty$ then on $F_\infty$ we obtain an SCO-system of partial shifts
$(\alpha_k)_{k \in \Nset_0}$, where the $\alpha_k$
are 
(for $n \in \Nset,\, k \in \Nset_0$) determined by
\[
\alpha_k^{(n)} :=
\left\{
\begin{array}{cc}
F(\delta^k): F[n-1] \rightarrow F[n] & \text{if}\;\; k=0,\ldots,n, \\ 
F(\delta^n): F[n-1] \rightarrow F[n] & \text{if}\;\; k>n\,.\;\quad\quad \\
\end{array} 
\right.
\]
We call this the SCO-system of partial shifts canonically associated to the SCO.

(b) Conversely, if $(\alpha_k)_{k \in \Nset_0}$ is an SCO-system of partial shifts for a filtration $(F_n)_{n\in\Nset_0}$ such that the $\mu_n \colon F_n \rightarrow F_\infty$ are monic
then defining $F[n] := F_n$ and 
\[
F(\delta^k) := \alpha_k^{(n)}: F[n-1] \rightarrow F[n] \quad \text{for} \; k=0,\ldots,n,\;n\in \Nset_0
\]
(where the $\alpha_k^{(n)}$ determine $\alpha_k$) yields an SCO and $(\alpha_k)_{k \in \Nset_0}$ is canonically associated to this SCO.
\end{theoremanddefinition}

\begin{proof}
In (a) it follows that for all $i,j \in \Nset_0$ with $i < j$
\[
\alpha_j^{(n+1)} \alpha_i^{(n)} 
= \alpha_i^{(n+1)} \alpha_{j-1}^{(n)}.
\]
In fact,
\begin{eqnarray*} 
\text{if}\; i \le n \;\text{and}\; j \le n+1: &\quad&
\alpha_j^{(n+1)} \alpha_i^{(n)} = F(\delta^j) F(\delta^i) 
= F(\delta^i) F(\delta^{j-1}) = \alpha_i^{(n+1)} \alpha_{j-1}^{(n)}, \\
\text{if}\; i > n \;\text{and}\; j > n+1: &\quad&
\alpha_j^{(n+1)} \alpha_i^{(n)} = F(\delta^{n+1}) F(\delta^n) 
= \alpha_i^{(n+1)} \alpha_{j-1}^{(n)}, \\
\text{if}\; i \le n \;\text{and}\; j > n+1: &\quad&
\alpha_j^{(n+1)} \alpha_i^{(n)} = F(\delta^{n+1}) F(\delta^i) 
= F(\delta^i) F(\delta^n) = \alpha_i^{(n+1)} \alpha_{j-1}^{(n)}\,. \\
\end{eqnarray*}

If the filtration $F[0] \stackrel{i_1}{\rightarrow} F[1] \stackrel{i_2}{\rightarrow} \ldots$ 
with $i_n := F(\delta^n) = \alpha_n^{(n)}: F[n-1] \rightarrow F[n]$, for $n \in \Nset$,
yields an $\omega-$colimit $F_\infty$ 
with morphisms $\mu_n: F[n] \rightarrow F_\infty$, for all $n \in \Nset_0$, satisfying $\mu_{n+1} i_{n+1} = \mu_n$ then
\[
i_{n+1} \alpha_k^{(n)} = \alpha_{n+1}^{(n+1)} \alpha_k^{(n)} = \alpha_k^{(n+1)} \alpha_n^{(n)}
= \alpha_k^{(n+1)} i_n
\]
and by applying $\mu_{n+1}$ we obtain
\[
\mu_n \alpha_k^{(n)} = \mu_{n+1} i_{n+1} \alpha_k^{(n)} = \mu_{n+1} \alpha_k^{(n+1)} i_n\,.
\]
From Lemma \ref{lem:adapted} we get, for all $k \in \Nset_0$, an adapted endomorphism $\alpha_k: F_\infty \rightarrow F_\infty$. 
We verify the properties of an SCO-system of partial shifts. First, for $k \in \Nset_0$
\[
\alpha_{k+1} \mu_k = \mu_{k+1} \alpha^{(k+1)}_{k+1}
= \mu_{k+1} i_{k+1} = \mu_k,
\]
which is property (1). Second, for all $n \in \Nset$ and $i<j$
\[
\alpha_j \alpha_i \mu_{n-1} = \alpha_j \mu_n \alpha_i^{(n)}
= \mu_{n+1} \alpha_j^{(n+1)} \alpha_i^{(n)}
= \mu_{n+1} \alpha_i^{(n+1)} \alpha_{j-1}^{(n)}
= \ldots = \alpha_i \alpha_{j-1} \mu_{n-1}\,.
\]
This implies 
$\alpha_j \alpha_i = \alpha_i \alpha_{j-1}$ if $i<j$ which is property (2). In fact, for $n \in \Nset$ we can define
\[
\mu^\prime_{n-1} := \alpha_j \alpha_i \mu_{n-1} = \alpha_i \alpha_{j-1} \mu_{n-1}
\]
and verify that
\[
\mu^\prime_n i_n = \alpha_j \alpha_i \mu_n i_n = \alpha_j \alpha_i \mu_{n-1} = \mu^\prime_{n-1}\,.
\]
It follows from the universal property that there is a unique morphism $\beta$ such that
$\mu^\prime_{n-1} = \beta \mu_{n-1}$ for all $n \in \Nset$. So $\beta = \alpha_j \alpha_i$ but with a similar argument also
$\beta = \alpha_i \alpha_{j-1}$ which proves our claim. 

Starting with (b) we just reverse the argument above to get (for all $0 \le i < j \le n+1,\; n\in \Nset$)
\[
\mu_{n+1} \alpha_j^{(n+1)} \alpha_i^{(n)}
= \alpha_j \alpha_i \mu_{n-1}
= \alpha_i \alpha_{j-1} \mu_{n-1}
= \mu_{n+1} \alpha_i^{(n+1)} \alpha_{j-1}^{(n)}\,.
\]
By assumption the $\mu_n$ are monic and we get $\alpha_j^{(n+1)} \alpha_i^{(n)} = \alpha_i^{(n+1)} \alpha_{j-1}^{(n)}$, which gives the cosimplicial identities for 
$F(\delta^k) := \alpha_k^{(n)}: F[n-1] \rightarrow F[n],\;\;
k=0,\ldots,n,\;n\in \Nset$. Hence if we apply the construction in (a) to this SCO then we get an SCO-system of partial shifts with the same $\alpha^{(n)}_k, k=0,\ldots,n,\;n\in \Nset$, as for the original system. 
But the remaining $\alpha^{(n)}_k$ with $k>n,\;n\in \Nset$, are also the same as for the original system. In fact, from property (1) of partial shifts we find (for $k>n,\;n\in \Nset$)
\[
\mu_n \alpha^{(n)}_k = \alpha_k \mu_{n-1} = \mu_{n-1} = \mu_n i_n 
\] 
and because the $\mu_n$ are monic this implies $\alpha^{(n)}_k = i_n$. We conclude that the original
sequence $(\alpha_k)_{k \in \Nset_0}$ is canonically associated to the SCO constructed from it.
\end{proof}

\begin{remark}\normalfont
Note that an SCO-system of partial shifts $(\alpha_k)_{k\in\Nset_0}$ canonically associated to an SCO satisfies the stronger local property (i.e., implying adaptedness) that the following diagram is commutative for all $n \in \Nset$ and $0\le k \le n$:
\[
\xymatrix{
F[n-1] \ar[d]_{\alpha^{(n)}_k} \ar[r]^{i_n} &F[n]\ar[d]^{\alpha^{(n+1)}_k}\\
F[n] \ar[r]_{i_{n+1}} &F[n+1]}
\]
Of course, this local property is also satisfied by any SCO-system of partial shifts with filtration $F_0 \stackrel{i_1}{\rightarrow} F_1 \stackrel{i_2}{\rightarrow} \cdots$, where each $\mu_n: F_n \rightarrow F_\infty$ is monic. In this case we are allowed to switch freely between SCOs and SCO-systems of partial shifts.
\end{remark}


In fact the length of the proof of Theorem \ref{thm:correspondence} above should not distract the reader from the fact that in all examples in this paper the correspondence is nothing but a rather direct assembling of all $\delta^k$ (for fixed $k$ and between different objects) into a single morphism $\alpha_k$ on the inductive limit. While in homological algebra SCOs are the natural starting point, in probability theory it is a common practice to study phenomena by constructing a big universe, i.e., a probability space common to all variables, and so it may be the SCO-system of partial shifts which first comes into view. This was indeed the case in the theory of (noncommutative) spreadability to which we apply our results in Section \ref{section:spread}. 
\\

Before developing some substantive connections with actions of the braid group and noncommutative probability in subsequent sections, we first give some examples by direct construction. 

\begin{example}\normalfont
It is worth noting that in the category of sets to each system of mappings $\alpha_k : X \rightarrow X, \; k\in\Nset_0$, satisfying $\alpha_j\alpha_i=\alpha_i\alpha_{j-1}$ for $i<j$, there is a canonical choice of filtration for which $(\alpha_k)_{k\in\Nset_0}$ is an SCO-system of partial shifts, as follows.

  Let $X$ be a set and for each $k \in\Nset_0$ let $\alpha_k : X \rightarrow X$ be a mapping. Furthermore, suppose that $\alpha_j\alpha_i = \alpha_i \alpha_{j-1}$ for $i,j \in \Nset_0,\, i<j$. For each $n \in\Nset_0$, let $X_n:=\{ x \in X : \alpha_{n+1}(x)=x \}$ (the fixed point set of $\alpha_{n+1}$) and let $i_n: X_{n-1} \rightarrow X_{n}$ be the inclusion mapping (which is well-defined since for $x \in X_{n-1}$ we have $\alpha_{n+1}(x)=\alpha_{n+1}\alpha_{n}(x) = \alpha_{n}\alpha_{n}(x)=x$). We will assume that $X$ is equal to the inductive limit $X_\infty:=\bigcup_{n\in\Nset_0} X_n$ (if it is not then we simply replace $X$ by $X_\infty$, after noting that $\alpha_n (X_\infty) \subset X_\infty$ for all $n\in\Nset_0$).  We claim that $(\alpha_k)_{k \in \Nset_0}$ is an SCO-system of partial shifts for the filtration $X_0 \stackrel{i_1}{\rightarrow} X_1 \stackrel{i_2}{\rightarrow} \cdots$.  First we see that each $\alpha_k$ is adapted by defining the mapping $\alpha^{(n)}_k : X_{n-1} \rightarrow X_n$ by $\alpha^{(n)}_k := (\alpha_k)|_{X_{n-1}}$ for all $n \in \Nset$. We see that condition (1) of Definition \ref{def:partial} is satisfied trivially and condition (2) of Definition \ref{def:partial} is given by assumption.
	
By easy modifications of these arguments such a canonical filtration based on fixed points can be obtained in many other categories, for example in the categories of noncommutative probability spaces considered in Section \ref{section:spread}.	
\end{example}

\begin{example}\label{example:groups}\normalfont
 For each $n \in \Nset_0$, let $G_n$ be a subgroup of the general linear group
$GL(n+1,R)$ over a unital ring $R$, such that 
$i_{n+1}(G_n) \subset G_{n+1}$. Here 
$i_{n+1}$ is the canonical embedding of $GL(n+1,R)$ in $GL(n+2,R)$, i.e., $i_{n+1}(g)=\begin{pmatrix}
      g & 0 \\ 0 & 1 \end{pmatrix}$
 for all $g \in GL(n+1,R)$, where each of the two zeros denotes a column or a row of $n+1$ zeros. We view $\Sset_{n+1}$ as the group of permutations on $\{0,1,\ldots,n\}$ and let $c_k$ denote the cycle $(k \;\; k+1 \;\; \cdots \;\; n)$. Let 
$\pi_{n+1}$ be the action of $\Sset_{n+1}$ on $GL(n+1,R)$ given by conjugation by permutation matrices, and suppose that $G_n$ is invariant under this action, for all $n \in \Nset_0$.  

      We can construct an SCO, $F$ say, in the category of groups by defining $F[n]:=G_n$ and $F(\delta^k):F[n-1] \rightarrow F[n]$ by $F(\delta^k)=\pi_{n+1}(c_k) i_n$ for $0\le k \le n$. [This means inserting a $k$-th row and column with a $1$ at the intersection and $0$s elsewhere. The cosimplicial identities are easy to check from that.]
			Through this construction we see, in particular, that $(GL(n,\Cset))_n,\;(U(n,\Cset)))_n,\;(SU(n,\Cset))_n$ and $(\Sset_n)_n$ are semi-cosimplicial groups, and it follows from Theorem 2.5 that we get SCO-systems of partial shifts on their inductive limits $GL(\infty,\Cset), U(\infty,\Cset),\; SU(\infty,\Cset)$ and $\Sset_\infty$, respectively.

      
For the symmetric groups $(\Sset_n)_{n\in\Nset}$ let us express the structure as a semi-cosimpli\-cial group in a more direct way. 
In this case we have $F[n] = G_n = \Sset_{n+1}$ for
$n \in \Nset_0$. We think of $F[0] = G_0 = \Sset_1$ as the trivial group while for $n \ge 1$ we have Coxeter generators
$\sigma_N := (N-1\; N)$ or star generators $\gamma_N := (0\; N)$, for $N=1,\ldots,n$ in both cases. 
Then we can check that $F(\delta^0) \sigma_N = \sigma_{N+1}$ for all $N$ while for $k \ge 1$
      \[
      F(\delta^k) (\gamma_N) = \begin{cases} \gamma_N & \text{if } N<k, \\ \gamma_{N+1} & \text{if } N \ge k. \end{cases}
      \]
The formula for $k \ge 1$ can be considered as an instance of Proposition \ref{prop:partial} for the relabeling $k \mapsto k-1$. 
\end{example}

Actually these examples of semi-cosimplicial groups belong to a general scheme of producing SCOs which we develop in its full generality in the following section.

\section{Semi-Cosimplicial Objects from Actions of the Braid Monoid $\Bset^+_\infty$}
\label{section:braid}

The braid groups $\Bset_n$ were introduced by Artin in \cite{Ar25}, see \cite{KT08} for a recent overview.  
For $n \ge 2$, $\;\Bset_n$ is presented by $n-1$ 
generators $\sigma_1,\ldots,\sigma_{n-1}$ satisfying the relations
\begin{align}
&&\sigma_i \sigma_{j} \sigma_i 
&= \sigma_{j} \sigma_i \sigma_{j} 
&\text{if $ \; \mid i-j \mid\, = 1 $;}&& \tag{B1}\label{eq:B1}\\
&&\sigma_i \sigma_j 
&= \sigma_j \sigma_i  
&\text{if $ \; \mid i-j \mid\, > 1 $.}&& \tag{B2}\label{eq:B2}
\end{align}
One has the inclusions $ \Bset_2 \subset \Bset_3 \subset \cdots \subset 
\Bset_{\infty}$, where  $\Bset_{\infty}$ denotes the inductive limit. 
The Artin generator $\sigma_i$ will be
presented as a geometric braid as follows:
\begin{figure}[h]
\setlength{\unitlength}{0.3mm}
\begin{picture}(360,35)
\savebox{\artin}(20,20)[1]{\masterartin} 
\savebox{\artininv}(20,20)[1]{\masterartininv} 
\savebox{\strandr}(20,20)[1]{\masterstrandr} 
\savebox{\strandl}(20,20)[1]{\masterstrandl} 
\savebox{\horizontaldots}(20,20)[1]{\masterhorizontaldots}
\put(100,00){\usebox{\strandl}}  
\put(100,00){\usebox{\strandr}}
\put(120,00){\usebox{\horizontaldots}}
\put(160,0){\usebox{\strandl}}
\put(180,0){\usebox{\artin}}
\put(200,0){\usebox{\strandr}}
\put(220,0){\usebox{\horizontaldots}}
\put(99,25){\footnotesize{$0$}}
\put(119,25){\footnotesize{$1$}}
\put(172,25){\footnotesize{$i-1$}}
\put(200,25){\footnotesize{$i$}}
\end{picture}
\end{figure}

Our convention in drawing diagrams of braids is that reading formulas from left to right corresponds to top-down compositions in the diagram.

It turns out that for the following arguments we do not need inverses of the Artin generators. Hence we consider $\Bset^+_\infty$, the monoid generated by $(\sigma_i)_{i \in \Nset}$.

Suppose that $\Bset^+_\infty$ acts on a set $X$, we simply write $gx \in X$ for the result of $g \in \Bset^+_\infty$ acting on $x \in X$. We define for $n \in \Zset, n \ge -1$
\[
X^n := \{ x \in X \colon \sigma_k\, x = x \;\,\text{if}\;\, k \ge n+2 \}
\]
which gives an increasing sequence $X^{-1} \subset X^0 \subset X^1 \subset \ldots$ of subsets of the set $X$.

\begin{theorem} \label{thm:braid}
$(X^n)_{n \ge -1}$ is an augmented semi-cosimplicial set (an augmented SCO in the category of sets), with the coface operators $\delta^k$ given by
\begin{eqnarray*}
\delta^k \colon \quad X^{n-1} &\rightarrow& X^n \quad\quad (k=0,\ldots,n,\quad n \in \Nset_0) \\
x &\mapsto& \sigma_{k+1} \ldots \sigma_{n+1} \,x \;.
\end{eqnarray*}
\end{theorem}

Note that $\sigma_{n+1} x = x$ for $x \in X^{n-1}$, so if $x \in X^{n-1}$ then for $k<n$ we can also write
$\delta^k \, x = \sigma_{k+1} \ldots \sigma_n \,x$ and for $k=n$ we have $\delta^n\,x = x$. Hence
$\delta^n: X^{n-1} \rightarrow X^n$ is nothing but the inclusion map, in particular this applies to the augmentation $\delta^0: X^{-1} \rightarrow X^0$. 

\begin{proof}
We use a double induction argument to prove
\[
\delta^j \delta^i = \delta^i \delta^{j-1} \colon X^{n-1} \rightarrow X^{n+1}
\]
for all $n \in \Nset_0$ and $i=0,\ldots,n,\;j=1,\ldots,n+1$ such that $i<j$. 
\\
We fix $n \in \Nset_0$. First suppose that $j=n+1$. If $i=n$ then for $x \in X^{n-1}$
\[
\delta^j \delta^i x = \delta^{n+1} \delta^n x = x = \delta^n \delta^n x = \delta^i \delta^{j-1} x\,.
\]
If, for all $x \in X^{n-1}$, the equation $\delta^j \delta^i x = \delta^i \delta^{j-1} x$ is valid for $j=n+1$
and for some $i$ with $1 \le i \le n$ then 
\[
\delta^{i-1} \delta^{j-1} x = \sigma_i \delta^i \delta^{j-1} x
\]
\[
= \sigma_i \delta^j \delta^i x = \delta^j \sigma_i \delta^i x
= \delta^j \delta^{i-1} x \,.
\]
We conclude by induction that for all $x \in X^{n-1}$ and $j=n+1$
the equation $\delta^j \delta^i x= \delta^i \delta^{j-1} x$ is valid for all $0 \le i \le n$. 
\\
Now suppose that for all $x \in X^{n-1}$ and some $j$ with $2 \le j \le n+1$ we have $\delta^j \delta^i x= \delta^i \delta^{j-1} x$ for all 
$0 \le i < j$. Then for $i < j-1$ 
\begin{eqnarray*}
\delta^{j-1} \delta^i x &=& \sigma_j \delta^j \delta^i x \;=\; \sigma_j \delta^i \delta^{j-1} x \\
&=& \sigma_j \sigma_{i+1} \ldots \sigma_{j-1} \sigma_j \sigma_{j+1} \ldots \sigma_{n+1} \sigma_j \ldots \sigma_{n+1} x \\
&=& \sigma_{i+1} \ldots \sigma_j \sigma_{j-1} \sigma_j \sigma_{j+1} \ldots \sigma_{n+1} \sigma_j \ldots \sigma_{n+1} x \\
&=& \sigma_{i+1} \ldots \sigma_{j-1} \sigma_j \sigma_{j-1} \sigma_{j+1} \ldots \sigma_{n+1} \sigma_j \ldots \sigma_{n+1} x \\
&=& \sigma_{i+1} \ldots \sigma_{n+1} \sigma_{j-1} \sigma_j \ldots \sigma_{n+1} x \;=\; \delta^i \delta^{j-2} x \,.
\end{eqnarray*}
(Here $\ldots$ always stands for $\sigma$'s with subscripts increasing by steps of $1$, including the case that we have the same $\sigma$ to the left and to the right of $\ldots\;$) By an induction argument for $j$ this proves the theorem.
\end{proof}

We remark that the theorem and the proof is still valid if we replace
the $X^n$ by any subsets $\tilde{X}^n \subset X^n$ so that
$\delta^k (\tilde{X}^{n-1}) \subset \tilde{X}^n$ is always satisfied.
\\

An alternative proof can be based on checking the following braid equalities:
\[
(\sigma_{j+1} \ldots \sigma_{n+1}) (\sigma_{i+1} \ldots \sigma_{n+1}) \sigma_{n+1}
= (\sigma_{i+1} \ldots \sigma_{n+1}) (\sigma_j \ldots \sigma_{n+1})
\] 
(for $0 \le i<j \le n$) in $\Bset^+_\infty$ (illustrated in the following diagram), together with
$\sigma_{n+1}\, x = x$ for $x \in X^{n-1}$.
\\
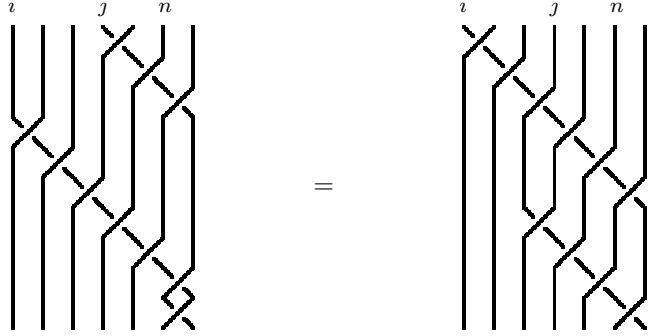
\begin{figure}[h]
\setlength{\unitlength}{0.2mm}
\begin{picture}(410,210)
\savebox{\artin}(20,20)[1]{\masterartin} 
\savebox{\artininv}(20,20)[1]{\masterartininv} 
\savebox{\strandr}(20,20)[1]{\masterstrandr} 
\savebox{\strandl}(20,20)[1]{\masterstrandl} 
\savebox{\horizontaldots}(20,20)[1]{\masterhorizontaldots}
\put(17,210){\footnotesize{$i$}}
\put(77,210){\footnotesize{$j$}}
\put(117,210){\footnotesize{$n$}}
\put(220,90){$=$}
\put(317,210){\footnotesize{$i$}}
\put(377,210){\footnotesize{$j$}}
\put(417,210){\footnotesize{$n$}}
\put(0,180){\usebox{\strandr}}
\put(20,180){\usebox{\strandr}}
\put(40,180){\usebox{\strandr}}
\put(60,180){\usebox{\strandl}}
\put(80,180){\usebox{\artin}}
\put(100,180){\usebox{\strandr}}
\put(120,180){\usebox{\strandr}}
\put(0,160){\usebox{\strandr}}
\put(20,160){\usebox{\strandr}}
\put(40,160){\usebox{\strandr}}
\put(60,160){\usebox{\strandr}}
\put(80,160){\usebox{\strandl}}
\put(100,160){\usebox{\artin}}
\put(120,160){\usebox{\strandr}}
\put(0,140){\usebox{\strandr}}
\put(20,140){\usebox{\strandr}}
\put(40,140){\usebox{\strandr}}
\put(60,140){\usebox{\strandr}}
\put(80,140){\usebox{\strandr}}
\put(100,140){\usebox{\strandl}}
\put(120,140){\usebox{\artin}}
\put(20,120){\usebox{\artin}}
\put(40,120){\usebox{\strandr}}
\put(60,120){\usebox{\strandr}}
\put(80,120){\usebox{\strandr}}
\put(100,120){\usebox{\strandr}}
\put(120,120){\usebox{\strandr}}
\put(0,100){\usebox{\strandr}}
\put(20,100){\usebox{\strandl}}
\put(40,100){\usebox{\artin}}
\put(60,100){\usebox{\strandr}}
\put(80,100){\usebox{\strandr}}
\put(100,100){\usebox{\strandr}}
\put(120,100){\usebox{\strandr}}
\put(0,80){\usebox{\strandr}}
\put(20,80){\usebox{\strandr}}
\put(40,80){\usebox{\strandl}}
\put(60,80){\usebox{\artin}}
\put(80,80){\usebox{\strandr}}
\put(100,80){\usebox{\strandr}}
\put(120,80){\usebox{\strandr}}
\put(0,60){\usebox{\strandr}}
\put(20,60){\usebox{\strandr}}
\put(40,60){\usebox{\strandr}}
\put(60,60){\usebox{\strandl}}
\put(80,60){\usebox{\artin}}
\put(100,60){\usebox{\strandr}}
\put(120,60){\usebox{\strandr}}
\put(0,40){\usebox{\strandr}}
\put(20,40){\usebox{\strandr}}
\put(40,40){\usebox{\strandr}}
\put(60,40){\usebox{\strandr}}
\put(80,40){\usebox{\strandl}}
\put(100,40){\usebox{\artin}}
\put(120,40){\usebox{\strandr}}
\put(0,20){\usebox{\strandr}}
\put(20,20){\usebox{\strandr}}
\put(40,20){\usebox{\strandr}}
\put(60,20){\usebox{\strandr}}
\put(80,20){\usebox{\strandr}}
\put(100,20){\usebox{\strandl}}
\put(120,20){\usebox{\artin}}
\put(0,0){\usebox{\strandr}}
\put(20,0){\usebox{\strandr}}
\put(40,0){\usebox{\strandr}}
\put(60,0){\usebox{\strandr}}
\put(80,0){\usebox{\strandr}}
\put(100,0){\usebox{\strandl}}
\put(120,0){\usebox{\artin}}
\put(320,180){\usebox{\artin}}
\put(340,180){\usebox{\strandr}}
\put(360,180){\usebox{\strandr}}
\put(380,180){\usebox{\strandr}}
\put(400,180){\usebox{\strandr}}
\put(420,180){\usebox{\strandr}}
\put(320,160){\usebox{\strandl}}
\put(340,160){\usebox{\artin}}
\put(360,160){\usebox{\strandr}}
\put(380,160){\usebox{\strandr}}
\put(400,160){\usebox{\strandr}}
\put(420,160){\usebox{\strandr}}
\put(320,140){\usebox{\strandl}}
\put(320,140){\usebox{\strandr}}
\put(340,140){\usebox{\strandl}}
\put(360,140){\usebox{\artin}}
\put(380,140){\usebox{\strandr}}
\put(400,140){\usebox{\strandr}}
\put(420,140){\usebox{\strandr}}
\put(320,120){\usebox{\strandl}}
\put(320,120){\usebox{\strandr}}
\put(340,120){\usebox{\strandr}}
\put(360,120){\usebox{\strandl}}
\put(380,120){\usebox{\artin}}
\put(400,120){\usebox{\strandr}}
\put(420,120){\usebox{\strandr}}
\put(320,100){\usebox{\strandl}}
\put(320,100){\usebox{\strandr}}
\put(340,100){\usebox{\strandr}}
\put(360,100){\usebox{\strandr}}
\put(380,100){\usebox{\strandl}}
\put(400,100){\usebox{\artin}}
\put(420,100){\usebox{\strandr}}
\put(320,80){\usebox{\strandl}}
\put(320,80){\usebox{\strandr}}
\put(340,80){\usebox{\strandr}}
\put(360,80){\usebox{\strandr}}
\put(380,80){\usebox{\strandr}}
\put(400,80){\usebox{\strandl}}
\put(420,80){\usebox{\artin}}
\put(300,60){\usebox{\strandr}}
\put(320,60){\usebox{\strandr}}
\put(340,60){\usebox{\strandl}}
\put(360,60){\usebox{\artin}}
\put(380,60){\usebox{\strandr}}
\put(400,60){\usebox{\strandr}}
\put(420,60){\usebox{\strandr}}
\put(300,40){\usebox{\strandr}}
\put(320,40){\usebox{\strandr}}
\put(340,40){\usebox{\strandr}}
\put(360,40){\usebox{\strandl}}
\put(380,40){\usebox{\artin}}
\put(400,40){\usebox{\strandr}}
\put(420,40){\usebox{\strandr}}
\put(300,20){\usebox{\strandr}}
\put(320,20){\usebox{\strandr}}
\put(340,20){\usebox{\strandr}}
\put(360,20){\usebox{\strandr}}
\put(380,20){\usebox{\strandl}}
\put(400,20){\usebox{\artin}}
\put(420,20){\usebox{\strandr}}
\put(300,0){\usebox{\strandr}}
\put(320,0){\usebox{\strandr}}
\put(340,0){\usebox{\strandr}}
\put(360,0){\usebox{\strandr}}
\put(380,0){\usebox{\strandr}}
\put(400,0){\usebox{\strandl}}
\put(420,0){\usebox{\artin}}
\end{picture}
\end{figure}

Note by looking at the diagram that the $i$-strand and the $j$-strand are not entangled with the other strands (which are always above them), but they are entangled with each other.
\\

Combining Theorem \ref{thm:braid} with Theorem \ref{thm:correspondence}
provides us with an SCO-system of partial shifts $(\alpha_k)_{k \in \Nset_0}$ canonically associated to an action of $\Bset^+_\infty$.
For an application of Proposition \ref{prop:partial} to such a situation in Example \ref{example:TL} we provide the following simplified formulas for powers of these partial shifts.

\begin{lemma} \label{lem:braid-partial}
If $x \in X^n \subset X$ (with $n \in \Nset_0$) then
for all $N \ge 1$
\[
(\alpha_n)^N (x) = \sigma_{n+N} \ldots \sigma_{n+1} x\,.
\]
\end{lemma}

\begin{proof}
From Theorem \ref{thm:braid} we have
\[
(\alpha_n)^N (x) = (\sigma_{n+1} \ldots \sigma_{n+N}) (\sigma_{n+1} \ldots \sigma_{n+N-1}) \ldots \sigma_{n+1} x
\]
(which is understood to be $\sigma_{n+1} x$ if $N=1$). 
This simplifies as shown above, as can be seen with an induction proof using the braid relations together with $x \in X^n$.
\end{proof}

We are mainly interested in situations where we have a (left) $\Bset^+_\infty$-module $V$ in which case Theorem \ref{thm:braid}
yields (at least) an augmented semi-cosimplicial abelian group. We give examples in a probabilistic setting in Section \ref{section:spread}. But let us give a few direct applications of Theorem \ref{thm:braid} immediately.

\begin{example}\normalfont 
It follows from Theorem \ref{thm:braid} that the sequence
$(\Bset_n)_n$ is a semi-cosimplicial group with the conjugation action of braids on themselves. In this case 
we choose $X := \Bset_\infty$ and $X^n := \Bset_{n+1}$
for all $n \in \Nset_0$ (and we define $X^0 = \Bset_1$ 
to be the trivial group). In fact, as required in our definition of $X^n$, we have 
\[
\Bset_{n+1} = \{ x \in \Bset_\infty \colon \;\sigma_k x \sigma^{-1}_k = x \;\; \text{for all}\;\; k \ge n+2 \}, 
\]
see for example Proposition 4.12 in \cite{GK09} for a proof.
Then for $x \in X^{n-1} = \Bset_n$ (with $n \in \Nset$) we have
\[
\delta^k(x) \quad:=\quad \sigma_{k+1} \ldots  \sigma_{n+1}\, x \,\sigma^{-1}_{n+1} \ldots  \sigma^{-1}_{k+1}.
\]
From the braid relations we can check that $\delta^0 (\sigma_N) = \sigma_{N+1}$ for all $N$.
If we use the so-called square roots of free generators $\gamma_1,\ldots,\gamma_n$ as generators for
$\Bset_{n+1}$ which are defined by
\[ 
\gamma_N := (\sigma_1 \ldots  \sigma_{N-1}) \sigma_N (\sigma^{-1}_{N-1} \ldots  \sigma^{-1}_1)
\]
then we have for $k \ge 1$ a direct way of describing the coface operators by
      \[
      \delta^k (\gamma_N) = \begin{cases} \gamma_N & \text{if } N<k, \\ \gamma_{N+1} & \text{if } N \ge k. \end{cases}
      \]
The formula for $k \ge 1$ can be considered as an instance of Proposition \ref{prop:partial} for the relabeling $k \mapsto k-1$. 
Again this can be checked by direct computations using the braid relations. Alternatively a detailed study of the so-called square roots of free generators presentation of the braid groups and of the corresponding partial shifts can be found in section 4 of \cite{GK09}. It is not an accident that this looks very similar to our example \ref{example:groups} with the sequence of symmetric groups considered as a semi-cosimplicial group. In fact,
it is instructive to check that we can go from the braid groups example to the symmetric groups example via the natural quotient map.
More generally, because the braid groups have the symmetric groups as quotients, we can always produce examples of SCOs from Theorem \ref{thm:braid} by actions of symmetric groups (interpreting them as actions of braid groups). The semi-cosimplicial groups produced in Example \ref{example:groups} are all of this type. 
\end{example}

\begin{example}\normalfont
For another class of examples we can consider solutions of Yang-Baxter equations. For illustration we choose the most basic setting: If $Y$ is a set and $r$ is a function from $Y \times Y$ to itself then $r$ is called a set-theoretic solution of the Yang-Baxter equation if on
$Y \times Y \times Y$ it satisfies $r^{12}\, r^{23}\, r^{12} = r^{23}\, r^{12}\, r^{23}$ where the superscript indicates on which copies $r$ acts. See for example \cite{GC12} for a recent investigation into such solutions. Clearly this defines an action of $\Bset^+_\infty$ on an infinite cartesian product $X$ of copies of $Y$ where $\sigma_k$ is represented by $r^{k-1,k}$.  
\end{example}

\begin{remark}\label{rem:cohom}\normalfont
Let us finally mention here that on $\Bset^+_\infty$-modules
we obtain among other things a version of the standard semi-cosimplicial cohomology theory which is always defined for SCOs in a module category. In fact, for all $n \in \Nset_0$ the differential
\[
d^n := \sum^n_{k=0} (-1)^k \delta^k \colon \; V^{n-1} \rightarrow V^n
\]
satisfies $d^{n+1} d^n=0$ and gives rise to the cohomology groups
\[
H^n := ker (d^{n+1}) / im(d^n)\,.
\]
Let us do a few direct computations for the SCOs produced from Theorem \ref{thm:braid}. On $V^{-1}$ we have $d^0 = \delta^0: x \mapsto x$. Further $d^1 = \delta^0 - \delta^1: V^0 \rightarrow V^1$, hence $d^1 x = \sigma_1 x - x$. It follows that both $im(d^0)$ and $ker(d^1)$ are equal to the fixed point set
of $\sigma_1$, so $H^0$ is trivial. Further $d^2 = \delta^0 - \delta^1 + \delta^2: V^1 \rightarrow V^2$, so
$d^2 x = \sigma_1 \sigma_2 x - \sigma_2 x + x$ for $x \in V^1$ and we find
\[
H^1 = \{x \in V^1: (\sigma_2-\sigma_1 \sigma_2)x = x\}
/ \{x \in V^1: x = \sigma_1 y - y \;\text{for}\;y \in V^0\}\,.
\]
An interpretation of these cohomology groups is not known
to us but it may be interesting to investigate when these groups are nontrivial and if they can play a role in the study of braid group representations. 

We also remark that other connections between the simplicial category and braid groups are investigated in the literature,
see for example \cite{Wu10},
but it is not clear to us how these investigations are related to our results above. 
\end{remark}

\section{Semi-Cosimplicial Objects and Spreadability in Noncommutative Probability}
\label{section:spread}

In this section we develop our theory within various categories of noncommutative probability spaces. We start in a very general situation and then by specializing make contact with settings that have a genuine probabilistic interpretation, as discussed in Section \ref{section:intro}. We refer to \cite{NS06} for further motivation to study these categories.

First consider a category with objects $(\CA,\varphi)$, where $\CA$ is a unital associative algebra over $\Cset$ 
and $\varphi: \CA \rightarrow \Cset$
is a linear functional with $\varphi(\eins)=1$ (i.e., unital), 
and with morphisms $\alpha: (\CA,\varphi) \rightarrow (\CB,\psi)$, where $\alpha$ is an algebra homomorphism satisfying $\alpha(\eins)=\eins$ (i.e., unital) and $\psi \circ \alpha = \varphi$. We call this the {\it category of noncommutative probability spaces} (as in \cite{NS06}). We mention at this point that there is no particular difficulty to work out the following theory in a non-unital setting but for definiteness we decided to concentrate on this standard version.

Let $(\CA,\varphi)$ be a noncommutative probability space. If $\CB$ is a unital associative algebra and $\iota: \CB \rightarrow \CA$ a unital algebra homomorphism then we can think of it as a morphism $\iota: (\CB,\varphi_B) \rightarrow (\CA,\varphi)$ with $\varphi_B := \varphi \circ \iota$. This is called a (noncommutative) random variable.  A sequence $(\iota_N)^\infty_{N=0}$ of such random variables is called a (noncommutative) random process and any expression of the form $\varphi(\iota_{N_1}(b_1),
\ldots, \iota_{N_k}(b_k)),\:b_i \in \CB$, (repetitions allowed), is called a moment of the process. If the variables do not commute we cannot speak of a joint distribution in the classical sense but there is the following replacement for it.
Let $\CA^f := *^\infty_{N=0} \CB$ be the (unital) free product of infinitely many copies of $\CB$ (see for example \cite{DKW14} for further uses of this construction in noncommutative probability) and let $\lambda_N: \CB \rightarrow \CA^f$ denote the canonical unital homomorphisms arising from this construction. The universal property ensures that there exists a unique unital 
homomorphism $\pi: \CA^f \rightarrow \CA$ such that $\pi \circ \lambda_N = \iota_N$ for all $N \in \Nset_0$. The unital linear functional $\varphi^f$ on $\CA^f$ defined by $\varphi^f := \varphi \circ \pi$ is called the distribution of the random process. Another way to think of a distribution is as a collection of all moments.

\begin{remark}\label{rem:process}\normalfont
In the literature it is often a sequence of elements $(x_N)_{N\in\Nset_0}$ in a noncommutative probability space which is called a noncommutative random process. This is just a special case of our setting where $\CB$ has a single generator and its image under $\iota_N$ is called $x_N$. The more flexible setting chosen here allows us to include multi-variable processes without much additional effort. 
\end{remark}

Processes with the same distribution are called stochastically equivalent. So if we are satisfied with stochastically equivalent versions it is possible to restrict our attention to free products: the processes $(\iota_n)$ and $(\lambda_n)$ given above have the same distribution if we endow $\CA^f$ with the functional $\varphi^f$. 

\begin{definition} \label{def:spread}
A sequence $(\iota_N)_{N\in \Nset_0}$ of unital homomorphisms from $\CB$ to the noncommutative probability space $(\CA, \varphi)$ is called spreadable if its distribution is unchanged if we pass to a subsequence, i.e., if for all $N \in \Nset_0$ we replace $\iota_N$ by $\iota_{i(N)}$
such that $N_1 < N_2$ implies $i(N_1) < i(N_2)$. 
\end{definition}

Spreadability is a distributional symmetry (or invariance principle). It only depends on the distribution and a sequence $(\iota_N)_{N\in \Nset_0}$ is spreadable if and only if the corresponding canonical sequence
$(\lambda_N)_{N\in \Nset_0}$ is spreadable (where $\CA^f$ is equipped with the functional $\varphi^f$ defined above). As indicated above, an equivalent description can be given in terms of moments where the definition of spreadability reduces to a system of equalities for numbers. 

We are ready for the main theorem. Informally stated, any appearance of SCOs in the category of noncommutative probability spaces always induces spreadability and, conversely, for any spreadable sequence we can always find an SCO which reproduces its distribution. Hence SCOs can be interpreted as the fundamental algebraic structure underlying spreadability. 

\begin{theorem} \label{thm:spread}
\begin{enumerate}
\item[]
\item Let an SCO be given in the category of noncommutative probability spaces with filtration $(\CA_n, \varphi_n)_{n\in \Nset_0}$ and
inductive limit $(\CA_\infty, \varphi_\infty)$. 
\\
Let $\iota_0 := \mu_0 \colon \CA_0 \rightarrow \CA_\infty$ and $\iota_N := (\alpha_0)^N \iota_0$ for $N \in \Nset_0$. 
Then $(\iota_N)_{N\in \Nset_0}$ is spreadable. 
(Here $\alpha_0$ is what we called the full shift among the partial shifts associated to the SCO. See Section \ref{section:adapted}.)
\item 
Let $(\iota_N)_{N\in \Nset_0}$ be a sequence of unital homomorphisms from the unital algebra $\CB$ to the noncommutative probability space $(\CA, \varphi)$ and let $(\CA^f, \varphi^f)$ be the corresponding (unital) free product equipped with the distribution (as described above, with $\lambda_N,\;N\in \Nset_0$, denoting the canonical embeddings, etc.). 
Consider the following statements $(a), (b), (c)$:
\begin{itemize}
\item[(a)]
$(\iota_N)_{N\in \Nset_0}$ is spreadable.
\item[(b)]
Let $\CA^f_n$ be generated by $\lambda_0(\CB), \ldots, \lambda_n(\CB)$ (as a unital algebra), for all $n \in \Nset_0$. Suppose that
$(\CA^f_n, \varphi^f_n)_{n\in\Nset_0}$ is an SCO in the category of noncommutative probability spaces with coface operators given by
\\
$\delta^k: (\CA^f_{n-1}, \varphi^f_{n-1}) \rightarrow (\CA^f_n,\varphi^f_n)$, for $k=0,\ldots,n$  (with $\varphi^f_n$ the restriction of $\varphi^f$ to $\CA^f_n$), determined (for $b \in \CB$) by
\[
\lambda_N (b) \mapsto
\left\{
\begin{array}{cc}
\lambda_N (b) & \text{if}\;\; N < k \\ 
\lambda_{N+1} (b) & \text{if}\;\; N \ge k. \\
\end{array} 
\right.
\]
\item[(c)]
Let $\CA_n$ be generated by $\iota_0(\CB), \ldots, \iota_n(\CB)$ (as a unital algebra), for all $n \in \Nset_0$. Suppose that
$(\CA_n, \varphi_n)_{n\in\Nset_0}$ is an SCO in the category of noncommutative probability spaces with coface operators
\\
$\delta^k: (\CA_{n-1}, \varphi_{n-1}) \rightarrow (\CA_n,\varphi_n)$, for $k=0,\ldots,n$  (with $\varphi_n$ the restriction of $\varphi$ to $\CA_n$), determined (for $b \in \CB$) by
\[
\iota_N (b) \mapsto
\left\{
\begin{array}{cc}
\iota_N (b) & \text{if}\;\; N < k \\ 
\iota_{N+1} (b) & \text{if}\;\; N \ge k. \\
\end{array} 
\right.
\]
\end{itemize}
Then $(a) \Leftrightarrow (b) \Leftarrow (c)$. 

\end{enumerate}
\end{theorem}

\begin{proof}
We start by proving (1).
Note that in the category of noncommutative probability spaces we can form inductive limits and hence we can go from SCOs to SCO-systems of partial shifts (as established in Section \ref{section:adapted}) whenever convenient. (The same applies to the $*$-setting studied later.) 

Let $\varphi(q)$ be a moment of a subsequence $(\iota_{i(N)})_{N\in \Nset_0}$. We define $p$ to be the finite product obtained by replacing
each factor $\iota_{i(N)}$ in $q$ by $\iota_N$. Suppose that the subscripts $N$ appearing in this way are $N_1 < N_2 < \ldots < N_R$. Let $M_1 := N_1$ and for $2 \le r \le R$ define 
$M_r  := N_r + [i(N_{r-1}) - N_{r-1}]$. 
Then $N_r \le M_r \le i(N_r)$ and,
making use of the properties of partial shifts stated in Proposition \ref{prop:partial}, we can check that
\[
\alpha^{i(N_R)-M_R}_{M_R} \ldots \alpha^{i(N_1)-M_1}_{M_1} (p) = q
\]
In fact, applying the partial shifts successively replaces variables of $p$ with the corresponding variables 
in $q$. Because the partial shifts preserve the functional $\varphi$ the proof is complete.

Now we prove the equivalence of (2)(a) and (2)(b). 
The formula in (b) always determines an algebra homomorphism 
$\delta^k$ between the free products $\CA^f_{n-1}$ and $\CA^f_n$,
for $k=0,\ldots,n$ and $n \in \Nset$. It is easily checked that these
$\delta^k$ satisfy the cosimplicial identities. (So this is always an SCO in the category of algebras.) 

If $(\iota_N)_{N\in \Nset_0}$ is spreadable then $\varphi^f_n \circ \delta^k = \varphi^f_{n-1}$ because we can always consider the subsequence which misses the $k$-th position. Hence (a) implies (b).
Conversely, given (b), the morphism $\delta^k$
(for $0 \le k \le n-1$) maps any polynomial in $\lambda_0(b_0), \ldots, \lambda_{n-1}(b_{n-1}),\;b_i \in \CB$,
in $\CA^f_{n-1}$ into the corresponding polynomial in $\lambda_0(b_0), \ldots, \lambda_{k-1}(b_{k-1}), 
\lambda_{k+1}(b_k), \ldots, \lambda_n(b_{n-1})$ in $\CA^f_n$.
(Note that $\delta^n: \CA^f_{n-1} \rightarrow \CA^f_n$ is nothing but the embedding of $\CA^f_{n-1}$ into $\CA^f_n$. Compare with our construction of the inductive limit from an SCO in Section \ref{section:adapted}.)
If $i: \Nset_0 \rightarrow \Nset_0$ is any strictly increasing function then by an induction argument we can always find a composition of coface operators which sends a polynomial $p$ in $\lambda_{N_1}(b_1), \ldots, \lambda_{N_R}(b_R)$ into the corresponding polynomial $q$ in $\lambda_{i(N_1)}(b_1), \ldots, \lambda_{i(N_R)}(b_R)$. (We have seen an explicit formula for this, using partial shifts, in the proof of (1).) Hence (b) implies (a). 

Finally we find that the implication from (2)(c) to 2(a) is in fact a special case of (1): here we identify the $\CA_n$ with their image in the inductive limit and omit the morphisms $\mu_n$ which are nothing but embeddings of the $\CA_n$ into $\CA_\infty \subset \CA$.
\end{proof}

We can now use our construction of SCOs from representations of braid monoids in Section \ref{section:braid} to produce many examples of spreadable sequences. This includes exchangeability which comes from representations of the symmetric groups. As an example reconsider the tensor product presented in Section \ref{section:intro}. The general case of exchangeability is characterized from this point of view in \cite{GK09}, Theorem 1.9. We can also
start with the semi-cosimplicial groups constructed in Sections
\ref{section:adapted} and \ref{section:braid} and obtain
SCOs in the category of noncommutative probability spaces based on the corresponding group algebras. For $\Bset_\infty$ the group von Neumann algebra
is studied from this point of view in \cite{GK09}, section 5. We
postpone a study of  other groups to future work.

\begin{example}\label{example:TL}\normalfont 
Instead
let us illustrate our theory here with an interesting example of spreadable sequences from the theory of subfactors in von Neumann algebras. For this we follow \cite{GHJ89},
in particular \textsection 4.4 there, where more details can be found. Note that to get a better fit with our previous notation we use a different numbering of the tower as in \cite{GHJ89}.  

Let $\CN \subset \CM$ be an inclusion of finite factors with finite Jones index $\beta = [\CM:\CN]$. Then with
$\CM_{-1} := \CN,\; \CM_0 := \CM$, Jones' basic construction yields a
tower 
\[
\CM_{-1} \subset \CM_0 \subset \CM_1 \subset \ldots
\]
and its weak closure with respect to the Markov trace $\trace$ yields the finite factor $\CM_\infty$. (To interpret the finite factor $\CM_\infty$ as an inductive limit we need a category of von Neumann algebras and von Neumann algebraic noncommutative probability spaces but for the following arguments we can also stay in the category of unital associative algebras and linear functionals introduced above where the inductive limit is just the union of all $M_k$.)
The algebra $\CM_\infty$ is generated (as a von Neumann algebra) by $\CM$ together with a sequence
of orthogonal projections $(e_n)_{n \in \Nset}$, called the Temperley-Lieb projections, satisfying $e_n \in \CM_n$ and
\[
e_n e_{n \pm 1} e_n = \beta^{-1} e_n,\quad e_n e_m = e_m e_n \; \text{if} \; |n-m| \ge 2
\]
(for all $n,m$). Then with $\beta = 2 + q + q^{-1}$ and defining
\[
g_n := q e_n - (\eins - e_n)
\]
it turns out that the $g_n$ satisfy the braid relations. So this determines a representation $\Bset_\infty \ni \sigma \mapsto g \in \CM_\infty$ by invertible elements $g \in\CM_\infty$ which in particular maps the Artin generator $\sigma_n$ to $g_n$, for all $n$.
Hence we can define an action of $\Bset_\infty$ on $\CM_\infty$ by $\sigma x := g x g^{-1}$ for $\sigma \in \Bset_\infty$ and $x \in \CM_\infty$. Clearly the Markov trace is invariant for this action (because it is a trace). Further note that $\CM = \CM_0$ commutes with $e_n$ and hence with $g_n$ for $n \ge 2$. From Theorem \ref{thm:braid} (and
Theorem \ref{thm:correspondence}) we get an SCO (and an SCO-system of partial shifts) and hence from Theorem \ref{thm:spread} we obtain spreadable sequences. Explicitly, with Lemma \ref{lem:braid-partial}, we conclude that 
the sequence $(\iota_n)_{n\in \Nset_0}$ of noncommutative random variables
$\iota_n \colon \CM \rightarrow \CM_\infty$ given by $\iota_0 := \id$ and for $n \ge 1$
\[
\iota_n := Ad (g_n \ldots g_1) 
\]
is spreadable. In particular
for any $x \in \CM$
the sequence $(x_N)_{N\in \Nset_0}$ given by
\[
x_0 := x, \; x_1 := g_1 x g^{-1}_1, \ldots,\; x_N := g_N \ldots g_1 x g^{-1}_1 \ldots g^{-1}_N, \ldots
\]
is spreadable, always with respect to the Markov trace. 

With the following modification we can find further spreadable sequences. For any $m \in \Nset_0$ consider the $m$-shifted action of
$\Bset_\infty$ on $\CM_\infty$ determined by $\sigma_n \mapsto 
g_{n+m}$, for all $n \in \Nset$. It follows that
the sequence $(\iota_n)_{n\in \Nset_0}$ of noncommutative random variables
$\iota_n \colon \CM_m \rightarrow \CM_\infty$ given by $\iota_0 := \id$ and for $n \ge 1$
\[
\iota_n := Ad (g_{m+n} \ldots g_{m+1}) 
\]
is spreadable.
In particular for any $x \in \CM_m$  the sequence $(x_N)_{N\in \Nset_0}$ given by
\[
x_0 := x, \; x_1 := g_{m+1} x g^{-1}_{m+1}, \ldots,\; x_N := g_{m+N} \ldots g_{m+1} x g^{-1}_{m+1} \ldots g^{-1}_{m+N}, \ldots
\]
is spreadable, with respect to the Markov trace, and
for all $m \in \Nset$ we find a spreadable sequence of projections $(e_{m,N})_{N \in \Nset_0}$ given by
\[
e_{m,0} := e_m, \; e_{m,1} := g_{m+1} e_m g^{-1}_{m+1}, \ldots,\; e_{m,N} := g_{m+N} \ldots g_{m+1} e_m g^{-1}_{m+1} \ldots g^{-1}_{m+N}, \ldots 
\]
Note that in general the braid group representations are not unitary and hence the coface operators and partial shifts in these arguments are algebra homomorphisms but not necessarily $*$-homomorphisms. This implies that the projections $e_{m,N}$ may be non-orthogonal projections. We further comment on this at the end of the section. 
\end{example}

Let us now, in the final part of this paper, go to $*$-algebras and to the probabilistic setting where the notion of spreadability originally comes from. Again this general setting can also be found in \cite{NS06}. If in a noncommutative probability space $(\CA,\varphi)$ as before $\CA$ is a (unital) $*$-algebra and
$\varphi$ is a unital positive linear functional (i.e., a state, positive in the sense that $\varphi(a^*a) \ge 0$ for all $a \in \CA$) then we call $(\CA,\varphi)$
a noncommutative $*$-{\it probability space}. We get the corresponding category by requiring morphisms $\alpha \colon (\CA,\varphi) \rightarrow (\CB,\psi)$ to be unital $*$-homomorphisms such that $\psi \circ \alpha = \varphi$. 
There is no need to repeat the definitions of random variables, moments, distributions, spreadability: they remain the same but now refer
to the new category of noncommutative $*$-probability spaces. In practice this can make a big difference. For example,
in the situation of a random process specified by a sequence $(x_N)_{N\in\Nset_0}$ of elements in a noncommutative probability space, see Remark \ref{rem:process}, if we work in the category of noncommutative $*$-probability spaces we have to take into account not only the elements $x_N$ themselves but also their adjoints $x^*_N$. To make the difference clear we talk about $*$-moments, $*$-distributions and $*$-spreadability but the reader should be aware that in the literature exclusively working in this setting the latter is usually just called spreadability.

It can be checked immediately that we can transfer our previous arguments to the category of noncommutative $*$-probability spaces and in this way successfully deal with $*$-spreadability.
For convenience we repeat Theorem \ref{thm:spread} explicitly in the $*$-setting and add a useful simplification which is available for faithful states. Recall that a positive functional $\varphi$ is called faithful if $\varphi(a^*a)=0$ for $a\in \CA$ implies $a=0$. 

\begin{theorem} \label{thm:spread*}
\begin{enumerate}
\item[]
\item Let an SCO be given in the category of noncommutative $*$-probability spaces with filtration $(\CA_n, \varphi_n)_{n\in \Nset_0}$ and
inductive limit $(\CA_\infty, \varphi_\infty)$. 
\\
Let $\iota_0 := \mu_0 \colon \CA_0 \rightarrow \CA_\infty$ and $\iota_N := (\alpha_0)^N \iota_0$ for $N \in \Nset_0$. 
Then $(\iota_N)_{N\in \Nset_0}$ is $*$-spreadable. 
(Here $\alpha_0$ is what we called the full shift among the partial shifts associated to the SCO.)
\item 
Let $(\iota_N)_{N\in \Nset_0}$ be a sequence of unital $*$-homomorphisms from the unital $*$-algebra $\CB$ to the noncommutative $*$-probability space $(\CA, \varphi)$ and let $(\CA^f, \varphi^f)$ be the corresponding (unital) free product equipped with the $*$-distribution (with $\lambda_N,\;N\in \Nset_0$, denoting the canonical embeddings, etc.). 
Consider the following statements $(a), (b), (c)$:
\begin{itemize}
\item[(a)]
$(\iota_N)_{N\in \Nset_0}$ is $*$-spreadable.
\item[(b)]
Let $\CA^f_n$ be generated by $\lambda_0(\CB), \ldots, \lambda_n(\CB)$ (as a unital $*$-algebra), for all $n \in \Nset_0$. Suppose that
$(\CA^f_n, \varphi^f_n)_{n\in\Nset_0}$ is an SCO in the category of noncommutative $*$-probability spaces with coface operators given by
\\
$\delta^k: (\CA^f_{n-1}, \varphi^f_{n-1}) \rightarrow (\CA^f_n,\varphi^f_n)$, for $k=0,\ldots,n$  (with $\varphi^f_n$ the restriction of $\varphi^f$ to $\CA^f_n$), determined (for $b \in \CB$) by
\[
\lambda_N (b) \mapsto
\left\{
\begin{array}{cc}
\lambda_N (b) & \text{if}\;\; N < k \\ 
\lambda_{N+1} (b) & \text{if}\;\; N \ge k. \\
\end{array} 
\right.
\]
\item[(c)]
Let $\CA_n$ be generated by $\iota_0(\CB), \ldots, \iota_n(\CB)$ (as a unital $*$-algebra), for all $n \in \Nset_0$. Suppose that
$(\CA_n, \varphi_n)_{n\in\Nset_0}$ is an SCO in the category of noncommutative $*$-probability spaces with coface operators given by
\\
$\delta^k: (\CA_{n-1}, \varphi_{n-1}) \rightarrow (\CA_n,\varphi_n)$, for $k=0,\ldots,n$  (with $\varphi_n$ the restriction of $\varphi$ to $\CA_n$), determined (for $b \in \CB$) by
\[
\iota_N (b) \mapsto
\left\{
\begin{array}{cc}
\iota_N (b) & \text{if}\;\; N < k \\ 
\iota_{N+1} (b) & \text{if}\;\; N \ge k. \\
\end{array} 
\right.
\]
\end{itemize}
Then $(a) \Leftrightarrow (b) \Leftarrow (c)$. If $\varphi$ is faithful then also $(a) \Rightarrow (c)$.

\end{enumerate}
\end{theorem}

\begin{proof}
All except the final statement about the faithful case can be proved by checking that the proof of Theorem \ref{thm:spread} can be adapted to the category of noncommutative $*$-probability spaces. Now assume that $\varphi$ is faithful and that $(\iota_N)_{N\in \Nset_0}$ is $*$-spreadable. If we try to define coface operators $\delta^k$ by the formulas in $(c)$ we note that these formulas guarantee the cosimplicial identities on the generators and hence on $\CA_{n-1}$ if we can extend these formulas to morphisms, necessarily in a unique way by the $*$-homomorphism property.
So we only need to check that extending the formulas given in $(c)$ for $\delta^k$ as morphisms is well defined. 
Indeed, if $p$ is any noncommutative polynomial in $\iota_0(b_0), \ldots, \iota_{n-1} (b_{n-1}),\;b_i \in \CB$, then by $*$-spreadability 
we find that $\varphi(\delta^k p) = \varphi(p)$ and
$\varphi(p^* p) = \varphi \big( (\delta^k p^*)(\delta^k p) \big)$, so $p =0$ implies $\delta^k p =0$ (because $\varphi$ is faithful). This shows that $\delta^k$ is well defined as a morphism in the category of $*$-probability spaces. 
\end{proof}

Note that all we need for the converse direction $(a) \Rightarrow (c)$ is the well-definedness of the morphisms $\delta^k$. The assumption of $\varphi$ being faithful is just a convenient sufficient condition to enforce that.

Let us finish by revisiting the spreadable sequences in towers of von Neumann algebras studied in Example \ref{example:TL}. This example shows that some care must be taken in distinguishing spreadability and $*$-spreadability.
As noted in \cite{GHJ89}, Example 4.2.10, in the case of small index $\beta \le 4$ the representations of $\Bset_\infty$ are unitary. This implies that the corresponding SCO-systems of partial shifts are given by $*$-endomorphisms and hence we are in the setting of Theorem \ref{thm:spread*}. We conclude that for small index these sequences are actually $*$-spreadable. Of course $*$-spreadability fits better into the category of von Neumann algebras as specific $*$-algebras and the whole theory of $*$-braidability developed in \cite{GK09} is now applicable in this situation.
It is less clear how to make good use of the spreadability in the case of big index $\beta > 4$, also proved above, when we cannot expect $*$-spreadability and the de Finetti type results of \cite{GK09,Ko10}, also discussed briefly in Section \ref{section:intro}, are no longer available. 

We refer back to the discussion in Section \ref{section:intro} for a wider view of  the importance of $*$-spreadability in (noncommutative) probability theory. A lot of further examples can be found in \cite{GK09}.

\end{document}